\documentclass{amsart}
\usepackage{amsmath, amssymb, amsthm, enumerate, mathrsfs, color, comment, xspace, cite}
\usepackage[colorlinks]{hyperref}
\hypersetup{linkcolor=blue, urlcolor=blue, citecolor=red}

\newcommand{\AND}{\qquad\text{and}\qquad}
\newcommand{\mat}[4]{\begin{pmatrix}#1&#2\\#3&#4\end{pmatrix}}
\newcommand{\tmat}[4]{\left(\begin{smallmatrix}#1&#2\\#3&#4\end{smallmatrix}\right)}
\newcommand{\COMMA}{,\quad}

\newcommand{\R}{\mathbb{R}}
\newcommand{\Q}{\mathbb{Q}}
\newcommand{\Z}{\mathbb{Z}}
\newcommand{\N}{\mathbb{N}}

\newcommand{\GAP}{\textsf{GAP}~\cite{GAP4}\xspace}

\numberwithin{equation}{section}
\newtheorem{thm}[equation]{Theorem}
\newtheorem{cor}[equation]{Corollary}
\newtheorem{conj}[equation]{Conjecture}
\theoremstyle{definition}

\newcommand{\libsemigroups}{\textsf{libsemigroups}~\cite{Mitchell2020aa}\xspace}

\title{Generating the monoid of $2 \times 2$ matrices over max-plus and
  min-plus semirings}
\author{James East, Julius Jonu\v sas, and J. D. Mitchell}
\date{\today}

\begin{document}

\begin{abstract}
  In this short note, we describe generating sets for the monoids of consisting
  of all $2 \times 2$ matrices over certain finite tropical semirings. 
\end{abstract}

\maketitle

A \textit{semigroup} is a set $S$ with an associative binary operation $*$. A
semigroup $S$ has an \textit{identity} if there exists $e\in S$ such that that
$e*s = s*e = s$ for all $s\in S$. A semigroup with identity is called a
\textit{monoid}.  The semigroup $S$ has a \textit{zero} if there exists $z\in
S$ such that $z*s=s*z=z$ for all $s\in S$. A semigroup $S$ is
\textit{commutative} if $s*t = t*s$ for all $s, t\in S$. 

A \textit{semiring} is a set $K$ equipped with two binary operations, which we
will denote by $\oplus$ and $\otimes$, and two distinguished elements $0, 1\in
K$,  such that the following conditions hold:

\begin{enumerate}[(i)]

  \item 
    $(K, \oplus)$ is a commutative monoid with identity $0$ (called the
    \textit{zero} of $K$);

  \item 
    $(K, \otimes)$ is a monoid with identity $1$ (called the \textit{one} of
    $K$);

  \item 
    $x\otimes (y \oplus z) = (x\otimes y) \oplus (x \otimes z)$ and 
    $(x \oplus y)\otimes z = (x\otimes z) \oplus (y \otimes z)$ for all $x, y,
    z \in K$;

  \item 
    $0\otimes x = x \otimes 0 = 0$ for all $x\in K$. 

\end{enumerate}

The simplest examples of semirings are rings themselves, such as $\Z$, $\Q$,
$\R$ with the usual addition and multiplication. An example of semiring that is
not a ring is the boolean semiring $\mathbb{B} = \{0, 1\}$ with operations
$\oplus$ and $\otimes$ defined by:
\begin{equation*}
  \begin{array}{c|cc}
    \oplus & 0 & 1 \\
    \hline
    0      & 0 & 1 \\
    1      & 1 & 1 
  \end{array}\AND
  \begin{array}{c|cc}
    \otimes & 0 & 1 \\
    \hline
    0       & 0 & 0 \\
    1       & 0 & 1 
  \end{array}.
\end{equation*}
The \textit{min-plus semiring} is defined to be the set $\N \cup \{\infty\}$
with the operations $\oplus = \min$ and $\otimes = +$, where $\otimes$ extends
the usual addition on $\N$ in a natural way, so that
$$x \otimes \infty = \infty \otimes x = \infty  \quad \text{for all}\quad x\in
\N\cup \{\infty\}.$$
We will denote the min-plus semiring by $K^{\infty}$. The one of $K^{\infty}$
is $0$ and the zero is $\infty$.  The min-plus semiring was first introduced by
Simon~\cite{Simon1978aa} in the context of automata theory; see also
Pin~\cite{Pinab}.  

The \textit{max-plus semiring} is defined to be the set $\N \cup \{-\infty\}$
with the operations $\oplus = \max$ and $\otimes = +$, where $\otimes$ extends
the usual addition on $\N$ in a natural way, so that $$x \otimes -\infty =
-\infty \otimes x = -\infty  \quad \text{for all}\quad x\in \N\cup
\{-\infty\}.$$ We will denote the max-plus semiring by $K^{-\infty}$. The one
of $K^{-\infty}$ is $0$ and the zero is $-\infty$.  A variant of the max-plus
semiring was first introduced by Mascle~\cite{Mascle1986} (Mascle's semiring
included both $\infty$ and $-\infty$).  The monoid of $2\times 2$ max-plus
matrices with entries in $\R\cup \{-\infty\}$ has also been studied in detail;
see~\cite{Johnson2011aa}.

Both the min-plus and max-plus semirings, and several others which we will not be
concerned with here, are referred to as \textit{tropical semirings}. 
A \textit{congruence} on a semiring $K$ is just an equivalence relation
$\rho\subseteq K\times K$ that is compatible with $\oplus$ and $\otimes$.
More specifically, if $(a, b), (c, d)\in\rho$, then 
$$(a \oplus c, b \oplus d), (a\otimes c, b\otimes d) \in \rho.$$

We are also concerned with the following finite
quotients of the semirings $K ^ {\infty}$ and $K ^ {-\infty}$:
$$K^{\infty}_t = K^{\infty} / (t = t + 1)\AND
  K^{-\infty}_t = K^{-\infty} / (t = t + 1)$$
where $t \in \N$ and $(t = t + 1)$ denotes the least congruence such that $t$
and $t + 1$ are equivalent. 

More explicitly, if $t\in \N$, then 
$$K^{\infty}_t = \{0, \ldots, t, \infty\}$$
and the operations on $K^{\infty}_t$ are defined by 
\begin{equation*}
  a \oplus b = \min\{a, b\}
  \AND
  a\otimes b = 
  \begin{cases}
    \min\{a + b,\ t\} & \text{if } a, b \not= \infty \\
    \infty            & \text{if } a\ \text{or}\ b = \infty.
  \end{cases}
\end{equation*}
Similarly, 
$$K^{-\infty}_t = \{0, \ldots, t, -\infty\}$$
and the operations on $K^{-\infty}_t$ are defined by 
\begin{equation*}
  a \oplus b = \max\{a, b\}\quad\text{and}\quad a\otimes b =  \min(a + b, t).
\end{equation*}

We denote by $M_n(K)$ the semigroup of $n \times n$ matrices with entries in a
semiring $(K, \oplus, \otimes)$, and the usual multiplication of matrices with
respect to $\oplus$ and $\otimes$.  Since every semiring $K$ has a one and a
zero, $M_n(S)$ is a monoid whose identity element is just the usual identity
matrix (with ones of the semiring on the diagonal and zeros elsewhere). 

We will find generating sets for $M_2(K^{\infty})$ and
$M_2(K^{-\infty})$, and hence also for their finite quotients
$M_2(K^{\infty}_t)$ and $M_2(K^{-\infty}_t)$ for all $t\in \N$. 

It is possible to compute subsemigroups of the monoids $M_n(K^{\infty}_t)$ and
$M_n(K^{-\infty}_t)$ defined by a generating set, for some relatively
small values of $n$ and $t$, using the C++ library \libsemigroups or 
Semigroupe~\cite{Pin2009aa} (a C program by
Jean-Eric Pin). 
The motivation for finding generating
sets for the entire monoids $M_n(K^{\infty}_t)$ and $M_n(K^{-\infty}_t)$ stems
from the second and third authors experiments when they reimplemented the
Froidure-Pin Algorithm~\cite{Froidure1997aa} in \libsemigroups.


\section{Min-plus matrices}

In this section, we are concerned with the monoid $M_2(K^{\infty})$ of $2\times
2$ matrices over the min-plus semiring $K^{\infty}$ and its finite quotients
$K^{\infty}_t$ for all $t\in \N$. 
The identity of $M_2(K^{\infty})$ is 
\begin{equation*}
  \begin{pmatrix}
    0      & \infty \\
    \infty & 0      
  \end{pmatrix}.
\end{equation*}

\begin{thm}
  \label{thm-min-plus}
  The monoid $M_2(K^{\infty})$ of $2 \times 2$ min-plus matrices is
  generated by the matrices:
  \begin{equation*}
    A_i = 
    \begin{pmatrix}
      i & 0      \\
      0 & \infty
    \end{pmatrix},
    \quad
    B = 
    \begin{pmatrix}
      1      & \infty \\
      \infty & 0
    \end{pmatrix},
    \quad 
    \text{and}
    \quad
    C = 
    \begin{pmatrix}
      \infty & \infty \\
      \infty & 0
    \end{pmatrix}
  \end{equation*}
  where $i\in \N\cup \{\infty\}$.
\end{thm}
\begin{proof}
  Pre-multiplying any matrix $X$ in $M_2(K^{\infty})$ by $A_{\infty}$ swaps the
  rows of $X$, and post-multiplying $X$ by $A_{\infty}$ swaps the columns of
  $X$.  Hence it suffices to express one representative of every matrix in 
  $M_2(K^{\infty})$ up to rearranging the rows and the columns. 

  It is routine to verify the following equalities:
  \begin{equation*}\label{eq-annihilate}
    \begin{pmatrix}
      a & b \\
      c & d
    \end{pmatrix}
    C 
    = 
    \begin{pmatrix}
      \infty & b \\
      \infty & d
    \end{pmatrix}
    \quad
    \text{and}
    \quad
    C 
    \begin{pmatrix}
      a & b \\
      c & d
    \end{pmatrix}
    = 
    \begin{pmatrix}
      \infty & \infty \\
      c      & d
    \end{pmatrix}
  \end{equation*}
  and 
  \begin{equation*}\label{eq-add-1}
    \begin{pmatrix}
      a & b \\
      c & d
    \end{pmatrix}
    B 
    = 
    \begin{pmatrix}
      a \otimes 1 & b \\
      c \otimes 1 & d
    \end{pmatrix}
    \quad
    \text{and}
    \quad
    B 
    \begin{pmatrix}
      a & b \\
      c & d
    \end{pmatrix}
    = 
    \begin{pmatrix}
      a \otimes 1 & b \otimes 1\\
      c     & d
    \end{pmatrix}.
  \end{equation*}

  If $i\in K^{\infty}$ is arbitrary, then 
  \begin{equation*}
    \begin{pmatrix}
      \infty & \infty \\
      \infty & i
    \end{pmatrix}
    = 
    C\ 
    \begin{pmatrix}
      \infty & 0 \\
      0      & i
    \end{pmatrix}\ 
    C
    = 
    C A_{\infty} A_i A_{\infty} C,
  \end{equation*}
  and in this way it is possible to generate every matrix in $M_2(K^{\infty})$
  containing three or four occurrences of $\infty$, using the given matrices.

  If $i, j\in K^{\infty}\setminus\{\infty\}$ and $k\in K^{\infty}$ are
  arbitrary, then 
  \begin{eqnarray*}
    \begin{pmatrix}
      \infty & i \\
      \infty & j
    \end{pmatrix}
    & = &
    B ^ i\ 
    \begin{pmatrix}
      \infty & 0 \\
      \infty & j
    \end{pmatrix}\ 
    = 
    B ^ i\ 
    \begin{pmatrix}
      \infty & 0 \\
      0      & j
    \end{pmatrix}\ 
    C
    = 
    B ^ i A_{\infty} A_j A_{\infty} C,
    \\
    \begin{pmatrix}
      \infty & \infty \\
      i      & j
    \end{pmatrix}
    & = &
    \begin{pmatrix}
      \infty & \infty \\
      0      & j
    \end{pmatrix}\ 
    B ^ i
    = 
    C
    \begin{pmatrix}
      \infty & 0 \\
      0      & j
    \end{pmatrix}\ 
    B ^ i\ 
    = 
    C A_{\infty} A_j A_{\infty} B ^ i,\\
    \begin{pmatrix}
      \infty & i \\
      j      & k
    \end{pmatrix}
    & = &
    B ^ i\ 
    \begin{pmatrix}
      \infty & 0 \\
      0      & k
    \end{pmatrix}\ 
    B ^ j
    = 
    B ^ i A_{\infty} A_k A_{\infty} B ^ j.
  \end{eqnarray*}
  (Note that the third equation gives $\tmat\infty i j \infty=B^iA_\infty
  B^j$.) Hence every matrix in $M_2(K^{\infty})$ with at least one occurrence
  of $\infty$ can be expressed as a product of the given matrices.

  Suppose that $a, b, c, d\in K^{\infty} \setminus \{\infty\}$. We may further
  suppose without loss of generality that $a = \min\{a,b,c,d\}$. Then 
  \begin{equation*}
    \begin{pmatrix}
      a & b \\
      c & d
    \end{pmatrix}
    = 
    \begin{cases}
      \begin{pmatrix}
        0     & 0 \\
        d - b & c- a
      \end{pmatrix}
       B ^ b A_{\infty} B ^ a & \text{if } b\leq d \\
      \begin{pmatrix}
        b - d & 0     \\
        0     & c - a 
      \end{pmatrix}
       B ^ d A_{\infty} B ^ a & \text{if } b > d. \\
    \end{cases}
  \end{equation*}
  It therefore suffices to note that if $i, j\in K^{\infty} \setminus
  \{\infty\}$ and
  $i \geq j$, then 
  \begin{equation*}
    \begin{pmatrix}
      0 & i \\
      j & 0
    \end{pmatrix}
    = 
    A_i A_j
    \quad \text{and} \quad
    \begin{pmatrix}
      0 & 0 \\
      i & j
    \end{pmatrix}
    = 
    A_{\infty} B ^ j A_{\infty} 
    \begin{pmatrix}
      0     & 0 \\
      i - j & 0
    \end{pmatrix}
    = 
    A_{\infty} B ^ j A_{\infty} A_0 A_{i - j}.\qedhere
  \end{equation*}
\end{proof}

\begin{cor}
  \label{cor-finite-min-plus}
  Let $t\in \N$ be arbitrary. Then the finite monoid $M_2(K^{\infty}_t)$ of $2
  \times 2$ min-plus matrices is generated by the $t + 4$ matrices:
  \begin{equation*}
    A_i = 
    \begin{pmatrix}
      i & 0      \\
      0 & \infty
    \end{pmatrix},
    \quad
    B = 
    \begin{pmatrix}
      1      & \infty \\
      \infty & 0
    \end{pmatrix},
    \quad 
    \text{and}
    \quad
    C = 
    \begin{pmatrix}
      \infty & \infty \\
      \infty & 0
    \end{pmatrix}
  \end{equation*}
  where $i\in \{0, 1, \ldots, t, \infty\}$.
\end{cor}

We do not know if the generating set in Corollary~\ref{cor-finite-min-plus} has
minimal size. However, it has been verified computationally for small values of
$t$, that the generating set in Corollary~\ref{cor-finite-min-plus} is
irredundant, in the sense none of the generators belongs to the semigroup
generated by the other generators. 

While knowing a generating set for the $2\times 2$ matrices is good, it would,
of course, be much better to have a generating set for the $n\times n$ matrices
for all $n \in \N$, $n > 2$. It seems unlikely that a case-by-case approach
(such as that performed in the proof of Theorem~\ref{thm-min-plus}) will be
successful for values of $n$ greater than $2$.  However, the following
conjecture is suggested by our computational experiments. The conjecture has
been verified when $t = 0, 1, \ldots, 5$.

\begin{conj}
  Let $t\in \N$ be arbitrary. Then the monoid $M_3(K^{\infty}_t)$ of $3 \times
  3$ min-plus matrices is generated by the $(2t^3 + 9t^2 + 19t + 36)/6$
  matrices:
  \begin{equation*}
    \begin{array}{l}
      \begin{pmatrix}
        \infty & \infty & 0      \\
        0      & \infty & \infty \\
        \infty & 0      & \infty \\
      \end{pmatrix},
      \quad
      \begin{pmatrix}
        \infty & \infty & 0      \\
        \infty & 0      & \infty \\
        0      & \infty & \infty \\
      \end{pmatrix},
      \quad
      \begin{pmatrix}
        \infty & \infty & \infty \\
        \infty & 0      & \infty \\
        \infty & \infty & 0      \\
      \end{pmatrix},
      \quad
      \begin{pmatrix}
        1      & \infty & \infty \\
        \infty & 0      & \infty \\
        \infty & \infty & 0      \\
      \end{pmatrix},
      \\
      \begin{pmatrix}
        0      & i      & \infty \\
        \infty & 0      & \infty \\
        \infty & \infty & 0      \\
      \end{pmatrix},
      \quad
      \begin{pmatrix}
        0      & i      & \infty \\
        \infty & 0      & i      \\
        i      & \infty & 0      \\
      \end{pmatrix}, 
      \quad
      i \in \{0, \ldots, t\}
      \\
      \begin{pmatrix}
        \infty & 0      & 0      \\
        0      & \infty & i      \\
        0      & j      & \infty \\
      \end{pmatrix}, 
      \quad
      i \in \{1, \ldots, t\},
      \quad 
      j \in \{1, \ldots, i\}
      \\
      \begin{pmatrix}
        0      & \infty & 0      \\
        i      & 0      & \infty \\
        \infty & j      & 0      \\
      \end{pmatrix}, 
      \quad
      i \in \{0, \ldots, t\},
      \quad 
      j \in \{1, \ldots, t\}
      \\
      \begin{pmatrix}
        0      & i      & \infty \\
        \infty & 0      & j      \\
        k      & \infty & 0      \\
      \end{pmatrix}, 
      \quad
      i \in \{1, \ldots, t\},
      \quad 
      j \in \{i, \ldots, t\},
      \quad 
      k \in \{1, \ldots, j - 1\}.
    \end{array}
  \end{equation*}
\end{conj}


\section{Max-plus matrices}

In this section, we are concerned with the monoid $M_2(K^{-\infty})$ of $2\times
2$ matrices over the max-plus semiring $K^{-\infty}$ and its finite quotients
$K^{-\infty}_t$ for all $t\in \N$. 

The identity of $M_2(K^{-\infty})$ is 
\begin{equation*}
  \mat0{-\infty}{-\infty}0.
\end{equation*}

\begin{thm}
  The monoid $M_2(K^{-\infty})$ of $2 \times 2$ max-plus matrices is
  generated by the matrices 
  \[
  X_i = \mat i00{-\infty} \COMMA
  Y = \mat 1{-\infty}{-\infty}0 \COMMA
  Z = \mat {-\infty}{-\infty}{-\infty}0 \COMMA
  W_{jk} = \mat 0jk0 ,
  \]
  where $i\in\N\cup \{-\infty\}$ and $j, k\in\N$ such that $0< j \leq k$.
\end{thm}

\begin{proof}
  Pre-multiplying any matrix $A$ in
  $M_2(K^{-\infty})$ by $X_{-\infty}$ swaps the
  rows of $A$, and post-multiplying $A$ by $X_{-\infty}$ swaps the columns of
  $A$.  Hence it suffices to express one representative of every matrix in 
  $M_2(K^{-\infty})$ up to rearranging the rows and the columns, as a product
  of the given generators. 

  It is routine to verify the following equalities:
  \begin{equation*}\label{eq-annihilate'}
    \begin{pmatrix}
      a & b \\
      c & d
    \end{pmatrix}
    Z 
    = 
    \begin{pmatrix}
      -\infty & b \\
      -\infty & d
    \end{pmatrix}
    \quad
    \text{and}
    \quad
    Z 
    \begin{pmatrix}
      a & b \\
      c & d
    \end{pmatrix}
    = 
    \begin{pmatrix}
      -\infty & -\infty \\
      c       & d 
    \end{pmatrix}
  \end{equation*}
  and 
  \begin{equation*}\label{eq-add-1'}
    \begin{pmatrix}
      a & b \\
      c & d
    \end{pmatrix}
    Y 
    = 
    \begin{pmatrix}
      a \otimes 1 & b \\
      c \otimes 1 & d
    \end{pmatrix}
    \quad
    \text{and}
    \quad
    Y 
    \begin{pmatrix}
      a & b \\
      c & d
    \end{pmatrix}
    = 
    \begin{pmatrix}
      a \otimes 1 & b \otimes 1\\
      c     & d
    \end{pmatrix}.
  \end{equation*}

  If $i\in K^{-\infty}$ is arbitrary, then 
  \begin{equation*}
    \begin{pmatrix}
      -\infty & -\infty \\
      -\infty & i
    \end{pmatrix}
    = 
    Z\ 
    \begin{pmatrix}
      -\infty & 0 \\
      0       & i
    \end{pmatrix}\ 
    Z
    = 
    ZX_{-\infty}X_iX_{-\infty}Z,
  \end{equation*}
  and in this way it is possible to generate every matrix in $M_2(K^{-\infty})$
  containing three or four occurrences of $-\infty$, using the given matrices.

  If $i, j\in K^{-\infty}\setminus\{-\infty\}$ and $k\in K^{-\infty}$ are
  arbitrary, then 
  \begin{eqnarray*}
    \mat {-\infty}{-\infty}ij & = & ZX_{-\infty}X_jX_{-\infty}Y^i,\\
    \mat {-\infty}i{-\infty}j & = & Y^iX_{-\infty}X_jX_{-\infty}Z, \\
    \mat{-\infty}ijk & = & Y^iX_{-\infty}X_kX_{-\infty}Y^j.
  \end{eqnarray*}
  Hence every matrix in $M_2(K^{-\infty})$ with at least one occurrence of
  $-\infty$ can be expressed as a product of the given matrices.

  Suppose that $a, b, c, d\in K^{-\infty} \setminus \{-\infty\}$. We may further
  suppose without loss of generality that $a = \min\{a, b, c, d\}$. Then 
  \begin{equation*}
    \begin{pmatrix}
      a & b \\
      c & d
    \end{pmatrix}
    = 
    \begin{cases}
      \begin{pmatrix}
        0     & 0     \\
        d - b & c - a
      \end{pmatrix}
       Y ^ b X_{-\infty} Y ^ a & \text{if } b\leq d \\
      \begin{pmatrix}
        b - d & 0     \\
        0     & c - a
      \end{pmatrix}
       Y ^ d X_{-\infty} Y ^ a & \text{if } b > d. \\
    \end{cases}
  \end{equation*}
  It therefore suffices to note that if $i, j\in K^{-\infty} \setminus
  \{-\infty\}$ and
  $i \geq j$, then 
  \begin{equation*}
    \begin{pmatrix}
      0 & i \\
      j & 0
    \end{pmatrix}
    = 
    W_{ij}
  \end{equation*}
  and 
  \begin{equation*}
    \begin{pmatrix}
      0 & 0 \\
      i & j
    \end{pmatrix}
    = 
    X_{-\infty}Y^j\mat{i-j}000 =
    X_{-\infty}Y^j\mat00{-\infty}0\mat{i-j}{-\infty}00,
  \end{equation*}
  with both $\tmat00{-\infty}0$ and $\tmat{i-j}{-\infty}00$ being a product of
  the given matrices, since they each have an entry equal to $-\infty$.
\end{proof}

\begin{cor}
  Let $t\in \N$ be arbitrary. Then the finite monoid $M_2(K^{-\infty}_t)$ of $2
  \times 2$ max-plus matrices is generated by the $(t^2+3t+8)/2$ matrices 
  \[
  X_i = \mat i00{-\infty} \COMMA
  Y = \mat 1{-\infty}{-\infty}0 \COMMA
  Z = \mat {-\infty}{-\infty}{-\infty}0 \COMMA
  W_{jk} = \mat 0jk0 ,
  \]
  where $i\in\{-\infty,0,\ldots,t\}$ and $j,k\in\{1,\ldots,t\}$ with $j\leq k$.
\end{cor}

Although we are able to produce generating sets for $M_3(K_t^{-\infty})$, for
some small values of $t$, using \GAP, we have not been able to recognise a
pattern in these matrices due to their large number and complexity. \GAP
produces irredundant generating sets for $M_3(K_t^{-\infty})$ with $19$, and
$78$ generators, respectively, when $t = 1$ and $2$.

\bibliography{tropical}{}
\bibliographystyle{plain}

\end{document}